\documentclass[11pt, letterpaper]{amsart}
\usepackage[top=3cm, bottom=4cm, left=2.5cm, right=2.5cm]{geometry}

\usepackage{verbatim}  
\usepackage{amsmath,amsthm,amsfonts,amssymb,amscd}
\usepackage{mathrsfs}
\usepackage{mathtools}
\usepackage[mathcal]{eucal} 
\usepackage{enumerate} 
\usepackage{mathrsfs} 
\usepackage{xcolor} 
\usepackage{float}
\usepackage{graphicx}
\usepackage{listings} 
\usepackage{hyperref}
\usepackage{todonotes}
\usepackage{bm}
\usepackage{derivative}
\usepackage{centernot}
\usepackage{array}
\usepackage{cellspace}
\usepackage{geometry}
\usepackage[skip=2pt plus1pt, indent=30pt]{parskip}

\hypersetup{%
  bookmarksnumbered=true,%
  colorlinks=true,%
  linkcolor=blue,%
  citecolor=blue,%
  filecolor=blue,%
  menucolor=blue,%
  urlcolor=blue,%
  pdfnewwindow=true,%
  pdfstartview=FitBH
  }

\newcommand{\x}{\times}

\newcommand{\T}{\top}

\newcommand{\bM}[1]{\begin{bmatrix}#1\end{bmatrix}}
 
\newcommand{\bsM}[1]{\begin{bsmallmatrix}#1\end{bsmallmatrix}}

\DeclarePairedDelimiter{\inr}{\langle}{\rangle}

\DeclareMathOperator{\mcg}{Mod}
\DeclareMathOperator{\tr}{tr}

\def\phi{\varphi}

\DeclarePairedDelimiter{\abs}{\lvert}{\rvert}

\def\SL{\mathrm{SL}}

\def\PSL{\mathrm{PSL}}

\def\bbR{\mathbb{R}}

\def\bbZ{\mathbb{Z}}

\newtheorem{thm}{Theorem}[section]
\newtheorem{prop}[thm]{Proposition}
\newtheorem{lem}[thm]{Lemma}
\newtheorem*{lem*}{Lemma}
\newtheorem{cor}[thm]{Corollary}

\theoremstyle{definition}

\newtheorem{rem}[thm]{Remark}

\numberwithin{equation}{subsection}
\graphicspath{ {./images/} }

\title{Thurston construction mapping classes with minimal dilatation}
\author{Maryam Contractor and Otto Reed}
\date{}
\begin{document}
\begin{abstract}
    Given a pair of filling curves $\alpha, \beta$ on a surface of genus $g$ with $n$ punctures, we explicitly compute the mapping classes realizing the minimal dilatation over all the pseudo-Anosov maps given by the Thurston construction on $\alpha,\beta$. We do so by solving for the minimal spectral radius in a congruence subgroup of $\PSL_2(\bbZ)$. We apply this result to realized lower bounds on intersection number between $\alpha$ and $\beta$ to give the minimal dilatation over any Thurston construction pA map on $\Sigma_{g,n}$ given by a filling pair $\alpha \cup \beta$. 
\end{abstract}
\maketitle 

\section{Statement of Results}

Let $\Sigma_{g,n}$ denote the surface of genus $g$ with $n$ punctures and let $\mcg(\Sigma_{g,n})$ denote the associated mapping class group. If $[f] \in \mcg(\Sigma_{g,n})$ is an isotopy class of pseudo-Anosov (pA) homeomorphisms of $\Sigma_{g,n}$, then there is an associated ``stretch factor'' $\lambda > 1$ which quantifies the scaling of its stable and unstable foliations (\cite{primer}, Section 13.2.3). This ``stretch factor'' or \emph{dilatation} $\lambda$ gives multiple perspectives of $f$. 

Among other things, $\lambda$ is the growth rate of the unstable foliation of $f$ under iteration and $\log(\lambda)$ is the topological entropy of $f$ (\cite{primer}, Theorem 13.2). In addition, there is a bijective correspondence between the set of dilatations in $\mcg(\Sigma_{g,n})$ and the length spectrum of closed geodesics in the moduli space of $\Sigma_{g,n}$. Finally, $\log(\lambda)$ gives the Teichm\"uller translation length, or the realized infimum distance that a point in Teichm\"uller space (under the Teichm\"uller metric) after action by $\mcg(\Sigma_{g,n})$. Thus finding minimal dilatation maps extends to minimizing entropy in subsets of the mapping class group, the length of closed geodesics in moduli space, and the Teichm\"uller translation length. 

There is literature on the problem of minimizing dilatation over pA maps in $\mcg(\Sigma_{g,n})$ (see \cite{farbleiningermargalit}), \cite{penner}). We consider a specific class of pA maps related to \emph{filling pairs} of curves in $\Sigma_{g,n}$. If $\alpha$ and $\beta$ are representatives of isotopy classes of simple closed curves $a$ and $b$ on $\Sigma_{g,n}$ and are in minimal position (i.e., the geometric intersection number of $a$ and $b$ equals $\abs{\alpha\cap\beta}$) we say $\alpha,\beta$ \emph{fill} $\Sigma_{g,n}$ if the complement $\Sigma_{g,n} \setminus (\alpha \cup \beta)$ is a union of topological disks or punctured disks. To any such filling pair $\alpha \cup \beta$, let $\Gamma_{\alpha,\beta}$ be the subgroup generated by Dehn twists about $\alpha$ and $\beta$. Thurston showed that any infinite order element of $\Gamma_{\alpha,\beta}$ not conjugate to a power of $T_\alpha$ or $T_\beta$ is pA (Theorem \ref{thurston construction}). Additionally, we call pseudo-Anosov elements of $\Gamma_{\alpha,\beta} \subset \mcg(\Sigma_{g,n})$ \emph{Thurston pA maps}. 
\begin{figure}[H]
    \centering
    \includegraphics[width=0.7\linewidth]{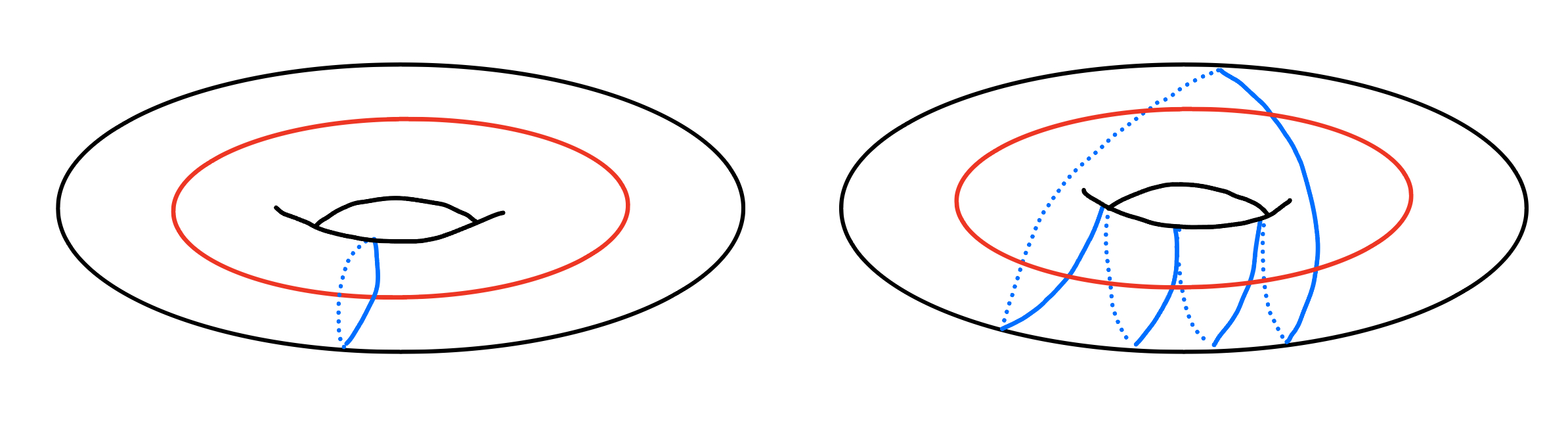}
    \caption{Two filling pairs on the torus: the second pair has greater geometric intersection number and consequently corresponds to a higher dilatation pA. Moreover, the composition of Dehn twists about the curves on the first torus is the minimal dilatation pA mapping class.}
\end{figure}
In this paper, we minimize dilatation all Thurston pA elements in $\Gamma_{\alpha,\beta}$ for any genus $g$ and number of punctures $n$ and find the following: 
\begin{thm}\label{main theorem}
    For $g\neq 0,2$, $n>2$ let $\alpha,\beta$ be any filling pair on $\Sigma_{g,n}$ and let $i(\alpha,\beta)$ be the geometric intersection number of $\alpha$ and $\beta$. Then the minimal dilatation over Thurston pA maps in $\Gamma_{\alpha,\beta}$ is achieved by the product $T_\alpha \cdot T_\beta$. This dilatation equals \[\frac{1}{2}((i(\alpha,\beta)^2+i(\alpha,\beta)\sqrt{(i(\alpha,\beta))^2-4}-2).\]
\end{thm}

We find that the minimum dilatation increases monotonically with the geometric intersection number $i(\alpha,\beta)$ for a filling pair $\alpha,\beta$. Using realized minimums for intersection number given by Aougab, Huang, and Taylor (\cite{aougab-huang}, Lemma 2.1-2.2, \cite{aougab-taylor}, Lemma 3.1, and summarized in section \ref{filling section}), we prove the following corollary giving a lower bound for the minimal dilatation Thurston pA map for all possible filling pairs. 

\begin{cor}\label{main corollary}
The minimal dilatation over all Thurston pA mapping classes in $\Gamma_{\alpha,\beta}$ for all filling pairs $\alpha,\beta$ in  $\Sigma_{g,n}$, $g\neq 0,2$ is given for $n = 0$ by 
\[\frac{1}{2}((2g-1)^2+(2g-1)\sqrt{(2g-1)^2-4}-2)\]
and for $n \geq 1$ by 
\[\frac{1}{2}((2g-1+n)^2+(2g-1+n)\sqrt{(2g-1+n)^2-4}-2).\] \\
Additionally, we have the following characterization: \\
\begin{center}
\def\arraystretch{1.5}
\begin{tabular}{|c|c|c|c|c|}  
    \hline
    Genus & Punctures  & $i(\alpha,\beta)$ & Minimal Dilatation Thurston pA\\
    \hline\hline
    $g = 0$ & $n \geq 4$ even & $n-2$ & $\frac{1}{2}((n-2)^2+(n-2)\sqrt{(n-2)^2-4}-2)$\\ \hline 
    $g = 0$ & $n$ odd & $n-1$ & $\frac{1}{2}((n-1)^2+(n-1)\sqrt{(n-1)^2-4}-2)$\\   
    \hline 
    $g = 2$ & $n \leq 2$ & $4$ & $7+4\sqrt{3}$\\ \hline 
    $g = 2$ & $n > 2$  & $2g+n-2$ & $\frac{1}{2}((2g+n-2)^2+(2g+n-2)\sqrt{(2g+n-2)^2-4}-2)$\\ \hline 
\end{tabular}
\end{center}
\end{cor}
\subsection{Proof idea.} To prove Theorem \ref{main theorem}, we use a theorem due to Thurston, which gives a representation into $\PSL_2(\bbR)$ for the subset of $\mcg(\Sigma_{g,n})$ generated by twists about filling pairs of curves (\cite{primer}, Section 14.1).\footnote{The construction also generalizes for multicurves, or disjoint collections of simple closed curves. Here we only present the Theorem \ref{thurston construction} for two filling curves--see \cite{primer}, Section 14.1 for the generalized version.} 

\begin{thm}[Thurston's Construction]\label{thurston construction}
    Suppose $\alpha,\beta$ are simple closed curves in $\Sigma_{g,n}$, $g,n\ge 0$ so that $\alpha \cup \beta$ fill $\Sigma_{g,n}$. Let $i(\alpha,\beta)$ denote geometric intersection number of $\alpha$ and $\beta$ and let $\Gamma_{\alpha,\beta}$ be the subgroup generated by Dehn twists $T_\alpha$ and $T_\beta$ about $\alpha$ and $\beta$, respectively. Then there is a representation $\rho: \Gamma_{\alpha,\beta} \rightarrow \PSL_2(\bbZ)$ given by \[T_\alpha \mapsto \bM{1 & -i(\alpha,\beta) \\ 0 & 1} \quad T_\beta \mapsto \bM{1 & 0 \\ i(\alpha,\beta) & 1}.\] Moreover, $\rho$ has the following properties: 
    \begin{enumerate}[(i)]
        \item There is a bijective correspondence between the sets of periodic, reducible, and pA elements in $\Gamma_{\alpha,\beta}$ and the sets of elliptic, parabolic, and hyperbolic elements in $\PSL_2(\bbR)$, respectively.
        \item Parabolic elements in $\rho(f)$ are exactly powers of $T_\alpha$ or $T_\beta$.
        \item If $\rho(f)$ is hyperbolic then the dilatation of $[f] \in \mcg(\Sigma_g)$ is exactly the spectral radius of $\rho(f)$.
    \end{enumerate}
\end{thm}

Using Thurston's representation, we minimize dilatation over all mapping classes in $\langle \rho(T_\alpha),\rho(T_\beta)\rangle \subseteq \PSL_2(\bbZ)$ to find the minimal dilatation mapping class in $\Gamma_{\alpha,\beta}$. Specifically, the smallest spectral radius matrices in the subgroup of $\SL_2(\bbZ)$ given by \[\Lambda_n \coloneq \left \langle \bM{1 & -n \\ 0 & 1}, \bM{1 & 0 \\ n & 1}\right \rangle,~n\ge 3\] achieve the dilatations given in Corollary \ref{main corollary}. 

\subsection{Comparison with prior literature.} There are interesting comparisons between our bounds in $\Gamma_{\alpha,\beta}$ and universal bounds for the entire mapping class group. 

Let $(\log(\lambda))_g$ denotes the smallest achieved topological entropy in $\mcg(\Sigma_{g,n})$ for $g \geq 2$, and define a relation $\sim$ on real valued functions $f, g$, where $f(x) \sim g(x)$ if there exists some constant $C> 0$ such that \[\frac{f(x)}{g(x)} \in [1/C,C], \quad \text{for all $x$}.\] Penner \cite{penner} gives that $(\log(\lambda))_g \sim \frac{1}{g}$. 

Thus in particular, as $g \rightarrow \infty$, there are pseudo-Anosov mapping classes whose dilatations get arbitrarily close to $1$. 

In contrast, in the most general case of our bound ($g \neq 0, 2$ and $n = 0$), the minimal dilatation in $\Gamma_{\alpha,\beta}$ increases monotonically with genus. This contrast with general bounds in $\mcg(\Sigma_{g,n})$ exemplifies that the proportion of pA maps in $\Gamma_{\alpha,\beta}$ to pA maps in $\mcg(\Sigma_{g,n})$ decreases as genus increases; thus, Thurston pA maps become increasingly less representative of the mapping class group for large genus. 
\newline

The authors would like to thank Benson Farb for posing the question that motivated this paper, continually supporting them for the duration of the project, and providing extensive comments on this paper, Faye Jackson for her invaluable explanations and intuition, and Amie Wilkinson for teaching a wonderful course in analysis where the authors first began collaborating. The authors would also like to thank Aaron Calderon for his patience in teaching them about entropy, Peter Huxford for his help on Proposition \ref{form of lambda}, Tarik Aougab for helpful remarks on Section \ref{filling section}, and Noah Caplinger for discussing Theorem~\ref{min dilatation}; this paper would not have been possible without their insight.

\section{Minimal spectral radii in $\Lambda_n$}

Recall we defined $\Lambda_n$ as \[\Lambda_n =\left \langle \bM{1 & -n \\ 0 & 1}, \bM{1 & 0 \\ n & 1}\right \rangle,~n\ge 3.\] The \emph{minimal dilatation} for any hyperbolic map in $\Lambda_n$ (and thus pA maps in $\Gamma_{\alpha,\beta}$) is given by \[\inf\{|\lambda(\alpha)|: |\lambda(\alpha)|>2, \alpha \in \Lambda_n\},\] where $\lambda(\alpha)$ is the spectral radius of $\alpha$. Since $\Lambda_n$ is discrete, this infimum must be realized. 

We begin with case when $n = 1$. In this case, the solution is well-known since $\Lambda_1\simeq SL_2(\mathbb{Z})$ (Theorem 2.5, \cite{primer}). So, the solution reduces to minimizing the roots of the characteristic polynomial equation \[x^2-\tr(\alpha)x + 1.\]

In $\SL_2(\bbZ)$, eigenvalues grow monotonically as a function of trace; the smallest magnitude trace in the infimum is $3$, so we have 
    \[x^2-3x+1=0 \implies \lambda = \frac{3+\sqrt{5}}{2}\] 

Now, finding $\alpha=\begin{bsmallmatrix} w & x \\ y & z\end{bsmallmatrix}$ follows immediately from the conditions $w + z = 3, wz-xy=1$: the solution is given by 
$\alpha = \begin{bsmallmatrix} 2 & 1 \\ 1 & 1 \end{bsmallmatrix}.$ Furthermore, $\alpha$ has two distinct real eigenvalues, so this solution is unique up to conjugacy.

For the general case, let 
\begin{align*}
A=\bM{1 & -n \\ 0 & 1}, \qquad B=\bM{1 & 0 \\ n & 1}
\end{align*}
We will assume $n \neq 2$; later (Remark \ref{no n equals 2}) we show that $\Lambda_2$ is not the representation given by the Thurston construction for any number $g$ or $n$. 

\begin{thm}\label{min dilatation}
    Fix $n > 2$. The minimal spectral radius in $\Lambda_n$ is given by $\frac{1}{2}(n^2+n\sqrt{n^2-4}-2)$ corresponding to the matrix $\begin{bsmallmatrix}
1-n^2 & -n \\ n & 1
\end{bsmallmatrix}$. 
\end{thm}

Fix $n > 2$. In $\PSL_2(\bbZ)$, the spectral radius of a matrix $\alpha$ is given by the larger root of the characteristic polynomial \[x^2-\tr(\alpha)x+1=0\] Explicitly, these solutions are \[x = \frac{\tr(\alpha)\pm \sqrt{(\tr(\alpha))^2-4}}{2}\] We wish to minimize spectral radius over hyperbolic matrices, so we assume also that $|\tr(\alpha)| > 2$. For $A \in \SL_2(\bbZ)$, it is also known that $\lambda(A)$ increases monotonically as a function of the magnitude of the trace; it follows that minimizing spectral radius is equivalent to minimizing trace magnitude. Here we minimize the latter and then compute the corresponding dilatation. 

To begin, we show the following, which was observed initially by Chorna, Geller and Shpilrain (Theorem 4(a), \cite{chorna}): \begin{prop}\label{form of lambda}
    Let $\alpha \in \Lambda_n$, $n > 2$. Then $\alpha$ has the form 
    \[\bM{1+k_1n^2 & k_2n \\ k_3n & 1+k_4n^2} \quad k_i \in \bbZ.\]
\end{prop}

\begin{proof}
    For simplicity, we say that a matrix $\gamma$ is \emph{congruent}, denoted \[\gamma \cong \bM{1 \mod n^2 & 0 \mod n \\ 0 \mod n & 1 \mod n^2},\] if $\gamma$ takes on the form 
    \begin{equation}\label{form of gamma}
        \gamma=\bM{1 + k_1n^2 & k_2n \\ k_3n & 1+k_4n^2}, \quad k_i \in \bbZ
    \end{equation} Define $S \subseteq \SL_2(\bbZ)$ as \[S\coloneqq \left\{\gamma \in \SL_2(\bbZ): \gamma \cong \bM{1 \mod n^2 & 0 \mod n \\ 0 \mod n & 1 \mod n^2}\right\}.\] We claim that $S$ is a subgroup of $\SL_2(\bbZ)$. Then, since $A, B \in \SL_2(\bbZ)$, it would follow that every $\gamma \in \Lambda_n$ would take on the form given by \ref{form of gamma}.

    To prove the claim, consider the natural homomorphism $\phi: \SL_2(\bbZ) \rightarrow \SL_2(\bbZ/n^2\bbZ)$ given by reduction modulo $n^2$. Then $S = \phi^{-1}(S')$, where \[S'\coloneqq \left\{\bM{1 & k_1n \\ k_2n & 1}: k_1, k_2 \in \SL_2(\bbZ/n^2\bbZ)\right\}.\] We show that $S'$ is a subgroup of $\SL_2(\bbZ/n^2\bbZ)$. Define $N, M \in \SL_2(\bbZ/n^2\bbZ)$ as \[N = \bM{1 & k_1n \\ k_2n & 1}, M = \bM{1 & k_3n \\ k_4n & 1}.\] Then we have \begin{align*}
        NM^{-1} &= \bM{1 & k_1n \\ k_2n & 1}\bM{1 & -k_3n \\ -k_4n & 1} \\ 
        & = \bM{1 -k_1k_4n^2 & n(k_1-k_3) \\ n(k_2-k_4) & 1-k_2k_3n^2} \\ 
        & \equiv \bM{1 & n(k_1-k_3) \\ n(k_2-k_4) & 1} \in S'. \\ 
    \end{align*} It follows that $S'$ is a subgroup of $\SL_2(\bbZ/n^2\bbZ)$. Then $S = \phi^{-1}(S')$, so $S$ is a subgroup of $\SL_2(\bbZ)$, giving the desired result. 
\end{proof}

\noindent \textit{Proof of Theorem \ref{min dilatation}} By Proposition \ref{form of lambda} it suffices to minimize trace over all matrices of the form \begin{equation}\label{general alpha}
    \alpha = \bM{k_1n^2+1 & k_2n \\ k_3n & k_4n^2+1} \quad \text{ such that } k_i \in \bbZ, (k_1n^2+1)(k_4n^2+1)-k_2k_3n^2=1
\end{equation} Note that we impose the second constraint equation because $\alpha \in \SL_2(\bbZ)$, so we have the determinant 
\[(k_1n^2+1)(k_4n^2+1)-k_2k_3n^2=1.\] 
Rearranging the determinant equation gives $k_2k_3=k_1k_4n^2+(k_1+k_4) \in \bbZ$. Thus, given any fixed $k_1, k_4 \in \bbZ$, there always exists $k_2, k_3$ such that the matrix $\bsM{1+k_1n^2 & k_2n \\ k_3n & 1+k_4n^2}$ is in $\SL_2(\bbZ)$. 

For any $\alpha$ given by \ref{general alpha}, $|\tr(\alpha)|$ is given by \[|2+n^2(k_1+k_4)|\] which is smallest when $k_1+k_4 = 0$. In this case $\tr(\alpha)=2$. Then $\alpha$ is not hyperbolic, so we disregard it. When $k_1+k_4=-1$, then $|\tr(\alpha)|=2-n^2$. For $k_1+k_4=1$, then $|\tr(\alpha)| = 2+n^2 > n^2-2$. Finally, for $|k_1+k_4|>1$, we have 
\[|2+n^2(k_1+k_4)| \in \{(k_1+k_4)n^2-2, (k_1+k_4)n^2+2\},\] 
which in either case is greater in magnitude than $n^2-2$. 

It is left to show that a matrix in $\Lambda_n$ achieves the minimum trace of $n^2-2$. Choosing $k_1 = -1, k_4 = 0$ gives the matrix $\bsM{1-n^2 & k_2n \\ k_3n & 1}$, which implies $k_2=-n$ and $k_3=n$. But this matrix is equal to $AB$, given by $\bsM{1 -n^2 & -n \\ n & 1}$. Thus both $AB$ and $BA$ (which are conjugate) in $\Lambda_n$ achieve the minimum dilatation of $\frac{1}{2}(n^2+n\sqrt{n^2-4}-2)$. 
\qed \newline 

To prove Theorem \ref{main theorem}, note that for two filling curves $\alpha, \beta$ on $\Sigma_{g,n}$ where $i(\alpha,\beta)=n$, we have $\Lambda_n = \langle \rho(T_\alpha), \rho(T_\beta)\rangle$ achieves its smallest dilatation map with $\rho(T_\alpha) \cdot \rho(T_\beta)$, since $\rho(T_\alpha) \cdot \rho(T_\beta)=AB \in \Lambda_n$. This map corresponds to $T_\alpha \cdot T_\beta$ in the associated mapping class group. 

\section{Construction of Filling Curves}\label{filling section}
We exposit work given by Aougab-Huang-Taylor \cite{aougab-huang}, \cite{aougab-taylor} and Jeffreys \cite{jeffreys}. For a fixed surface $\Sigma_{g,n}$, our goal is to obtain a lower bound for the intersection number of a pair of filling curves and subsequently construct examples achieving these minima. We use lower bounds given by the filling permutations of Aougab-Huang \cite{aougab-huang} and Aougab-Taylor \cite{aougab-taylor} and the generalized filling permutations of Jeffreys \cite{jeffreys}, which gives us an algebraic way to describe ``gluing patterns" of polygons. 

The idea is to construct polygons whose sides are identified in such a way that, once glued, they form the surface $\Sigma_{g,n}$ with the glued sides becoming the filling curves $\alpha,\beta$. Each polygon will correspond to a disk in the complement of $\alpha\cup \beta$ on $\Sigma_{g,n}$, so we can retroactively puncture the polygons to form $\Sigma_{g,n}$. Since we will ``place" the punctures, our convention will be to treat them as marked points and thus exclude them from the Euler characteristic.

We begin with a general lower bound for the intersection number on any surface $\Sigma_{g,n}$ from Aougab-Huang (\cite{aougab-huang}, Lemma 2.1).
\begin{lem}\label{lower bound}
    Fix $g \geq 1, n \geq 0$. If $\alpha, \beta$ fill $\Sigma_{g,n}$, then $i(\alpha,\beta) \geq 2g-1$, where $i$ denotes geometric intersection number. 
\end{lem}
\begin{proof}
    We model $\alpha,\beta$ as a $4$-valent graph $G$ (where vertices $v$ are intersection points) since the complement $\Sigma_g \setminus (\alpha,\beta)$ is a union of topological discs $D$. The Euler characteristic of the graph must match that of $\Sigma_{g,n}$. We know \[\sum_{v \in G} \deg_v(G) = 2|E| = 4|V| = 2i(\alpha,\beta)\] Then we obtain \[\chi(\Sigma_g) = 2-2g = |D|-2i(\alpha,\beta)+i(\alpha,\beta)\] and since $|D| \geq 1$, we have the result. 
\end{proof}

This bound is only realized in the case when $n=0$. For punctured surfaces, however, we can come very close. To construct an explicit example where equality is realized, we now introduce the notion of \emph{filling permutations} from \cite{aougab-huang} and \cite{jeffreys}.

Fix a surface $\Sigma_{g,n}$ and a filling pair $\alpha,\beta$. We will label the subarcs of the curves (segments connecting two intersection points) in the following manner, beginning with the curve $\alpha$. Fix an orientation for $\alpha$ and choose a starting intersection point $x_0\in\alpha\cap\beta$. Travel in the direction of $\alpha$ until we reach an intersection point $x_1\neq x_0$ and label the subarc of $\alpha$ joining $x_0$ to $x_1$ as $\alpha_1$. Continue this process until we arrive back at $x_0$--this will occur since the curve $\alpha$ is closed--labeling the subarcs $\{\alpha_1,\ldots,\alpha_m\}$; note that $m=i(\alpha,\beta)$. Repeat this process with $\beta$ to obtain a labeling $\{\beta_1,\ldots,\beta_m\}$.

Now, cutting the surface along $\alpha\cup\beta$, we obtain $n+2-2g$ polygons whose sides correspond to subarcs of $\alpha$ and $\beta$ and whose vertices are intersection points in $\alpha\cap\beta$. Orient these polygons clockwise. Our goal is to describe the polygons algebraically in terms of permutations acting on their edges. First note that since we cut along $\alpha$ and $\beta$ to obtain these polygons, every subarc $\alpha_k$ of $\alpha$ will have an inverse $\alpha_k^{-1}$ with the opposite orientation; similarly for $\beta$. Define
\[
    A=\{\alpha_1,\beta_1,\ldots,\alpha_m,\beta_m,\alpha^{-1},\beta^{-1},\ldots,\alpha_m^{-1},\beta_m^{-1}\}
\]
and identity $A$ with the set $\{1,\ldots,4m\}$. Label the sides of the polygons with the corresponding elements of $A$.

Now, we define the \emph{filling permutation} of a polygon as $\sigma(j)=k$ if, and only if, traveling clockwise around the polygon the edge labeled by the $j$th element of $A$ is followed by the edge labeled by the $k$th element of $A$. Each filling permutation will be a cycle in $S_{4m}$, the symmetric group on $4m$ elements, so since there are $n+2-2g$ polygons we have $n+2-2g$ corresponding cycles $\sigma\in S_{4m}$. 
\begin{figure}[H]
    \centering
    \includegraphics[width=0.75\linewidth]{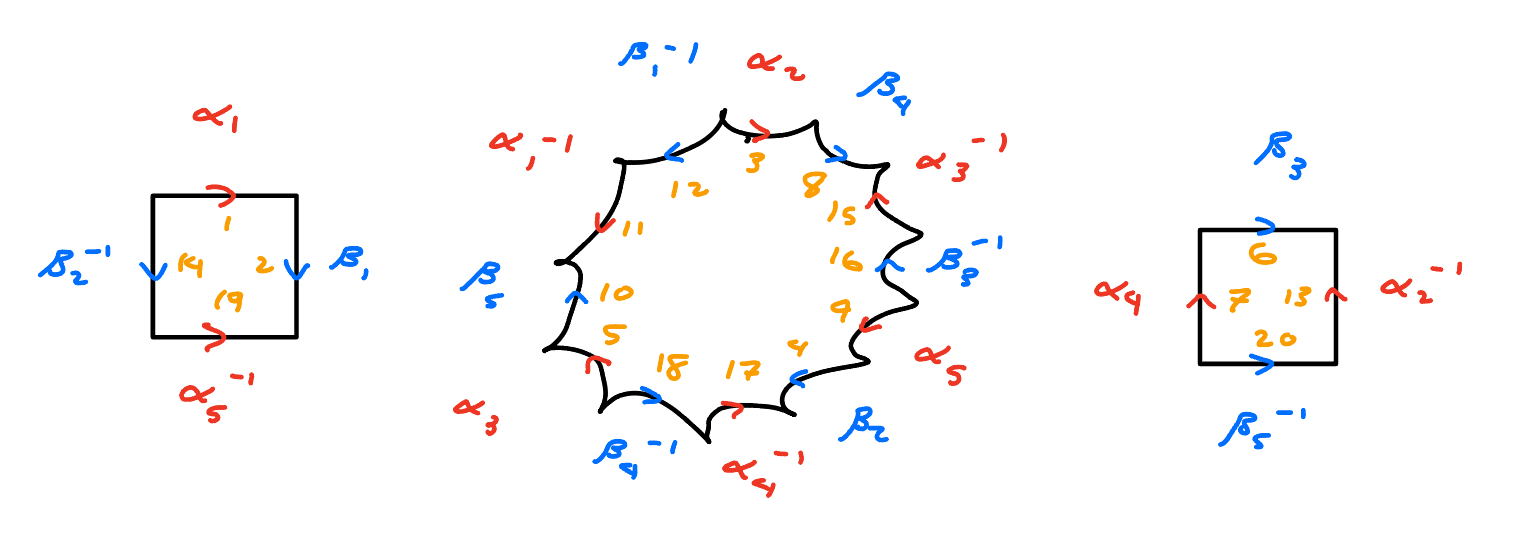}
\caption{The polygons corresponding to a filling pair on $\Sigma_{2,3}$. The associated filling permutations are, from left to right, $(1,2,19,14),$ $(3,8,15,16,9,17,18,5,10,11,12)$, and $(6,13,20,7)$}
\end{figure}

There are two more geometrically significant permutations we are interested in. Take $Q=Q_{\alpha,\beta} \in S_{4m}$ as $Q = (1, 2, \dots, 4m)^{2m}$. We note that $Q$ acts on the edges by reversing their orientation, i.e., it sends $j$ to $k$ if and only if the $j$th and $k$th elements of $A$ are inverses of each other. Finally, define $\tau=\tau_{\alpha,\beta} \in S_{4m}$ as
\[
    \tau=(1,3,5,\ldots,2m-1)(2,4,6,\ldots,2m)(4m-1,4m-3,\ldots,2m+1)(4m,4m-2,\ldots,2m+2).
\]
The first cycle represents sending $\alpha_i$ to $\alpha_{i+1}$, the second $\beta_i$ to $\beta_{i+1}$, the third $\alpha^{-1}_k$ to $\alpha^{-1}_{k+1}$ and the fourth $\beta^{-1}_{k}$ to $\beta^{-1}_{k+1}$. In other words, $\tau$ moves each arc in $\alpha$ to the next arc of $\alpha$ with the same orientation, and similarly for $\beta$. 

We will say that a permutation is \emph{parity-respecting} if it sends even numbers to even numbers and odd numbers to odd numbers and \emph{parity-reversing} if it sends even numbers to odd numbers and odd numbers to even numbers.

The following lemma from Jeffreys (\cite{jeffreys}, Lemma 2.3) gives the conditions necessary to define a filling permutation on a surface $\Sigma_{g,n}$; we will subsequently construct the filling curves by finding a permutation that satisfies these hypotheses.
\begin{lem}\label{construction}
    Let $\alpha,\beta$ be a filling pair on $\Sigma_{g,n}$ with $i(\alpha,\beta)=m\ge i(\alpha,\beta)$, the minimal intersection number. Then, $\sigma=\sigma_{\alpha,\beta}$ satisfies $\sigma Q \sigma=\tau$. Conversely, a parity-reversing permutation $\sigma\in S_{4m}$ consisting of $m+2-2g$ cycles and no more than $n$ 2-cycles that satisfies the above relation defines a filling pair on $\Sigma_{g,n}$ with intersection number $m$.
\end{lem}
Now we have the necessary ingredients to compute the minimal realized number of intersection points on $\Sigma_{g,n}$; we closely follow the one given by \cite{aougab-taylor}, Lemma 3.1.
\begin{prop} \label{intersection numbers}
    Suppose $g \neq 0,2$ and $n = 0$. If $\alpha,\beta$ are minimally intersecting filling curves on $\Sigma_{g,n}$, then \[i(\alpha,\beta)=2g-1.\] If $n \ge 1$, then  \[i(\alpha, \beta)=2g+n-2.\]
\end{prop}

\begin{proof}
    Using the same argument as in Lemma \ref{lower bound}, we have that $i(\alpha,\beta)=2g+n-2+|D|$ where $|D|$ is the number of topological disks in the complement of $\alpha \cup \beta$ in $\Sigma_{g,n}$. Thus we have the lower bounds and it is left to show that these bounds are realized. The first case is given explicitly by Lemma \ref{construction}; for the second, we induct on $n$. When $n = 1$, then $2g-1=2g+n-2$. Thus the filling curves given in Lemma \ref{construction}, which have a single disk $D$ in their complement, still fill $\Sigma_{g,1}$, obtained by puncturing $D$ once. 
    
    To begin constructing the filling pairs for surfaces with $n \geq 1$ punctures, we give an example for when $g = 1$; there is a formula for the intersection number of curves on the torus (\cite{primer}, Section 1.2.3). Namely, if $\alpha$ is a $(p,q)$ curve and $\beta$ is an $(r,s)$ curve, then \[i(\alpha,\beta)=ps-qr.\] Taking $\alpha$ to be a $(n,1)$-curve and $\beta$ to be a $(0,1)$ gives two curves intersecting exactly $n$ times. The complement of these two curves is $n$ topological disks, and puncturing each gives $2g+n-2=n$ intersections on $\Sigma_{1,n}$. 

    Now, we describe the \textit{double bigon method,} which begins with a filling pair $\alpha,\beta$ on $\Sigma_{g,n}$ and constructs a filling pair on $\Sigma_{g,n+2}$ with intersection number $i(\alpha,\beta)+2$. As before, let $\alpha,\beta$ be a filling pair on $\Sigma_{g,n}$, and orient and label them into subarcs $\alpha_1, \dots, \alpha_{i(\alpha,\beta)}$ and $\beta_1, \dots, \beta_{i(\alpha,\beta)}$. Suppose $i(\alpha_1,\beta_{i(\alpha,\beta)} \neq 0$. Then pushing $\alpha_1$ across $\beta_{i(\alpha,\beta)}$ and back over forms $2$ bigons. Puncturing each of these bigons  gives the same pair of filling curves on $\Sigma_{g,n+2}$ with intersection number $i(\alpha,\beta)+2$. See Figure \ref{double bigon} for reference. 
    \begin{figure}[h]
    \centering
    \includegraphics[width=0.7\linewidth]{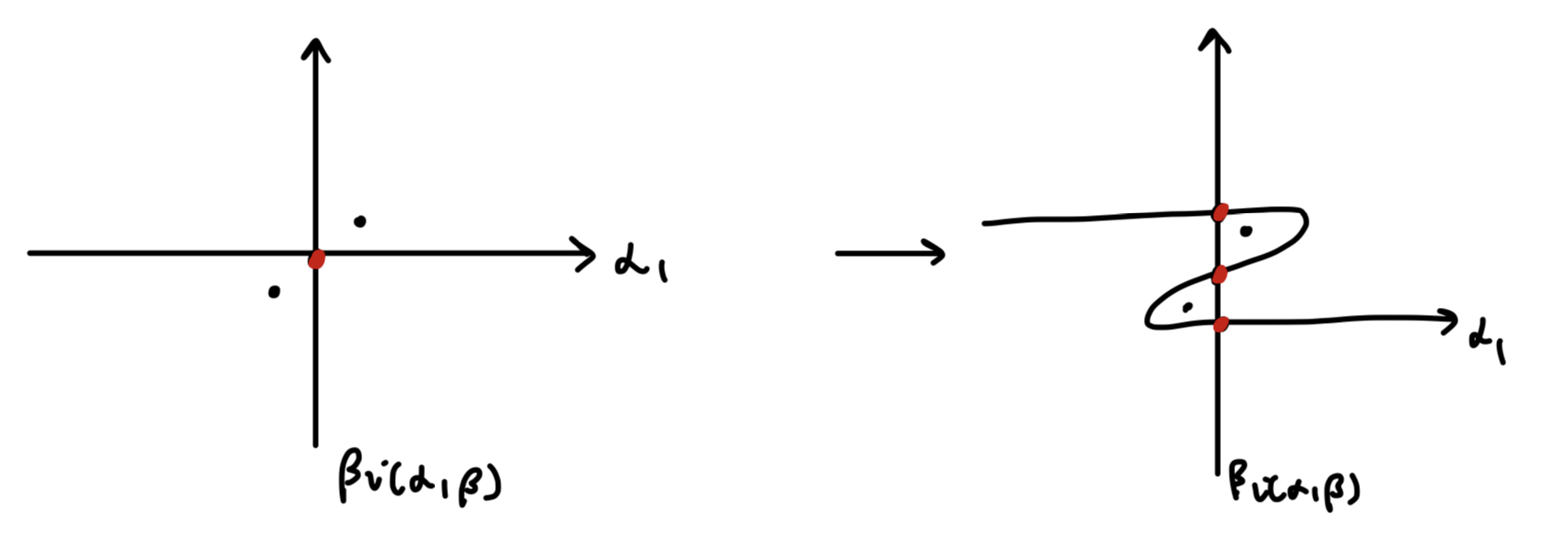}
    \caption{The ``double bigon method.'' Given a pair of filling curves $\alpha,\beta$ on a surface $\Sigma_{g,n}$ with intersection number $i(\alpha,\beta)$, the same pair fills $\Sigma_{g,n+2}$ with intersection number $i(\alpha,\beta)+2$.}
    \label{double bigon}
\end{figure}

    Suppose $n = 2k+1$ is odd and $g > 2$. Take a pair $\alpha,\beta$ which fill $\Sigma_{g,0}$, whose complement is connected, i.e. is a single topological disk, and such that $i(\alpha,\beta)=2g-1$ (we know such an $\alpha,\beta$ exist by Lemma \ref{construction}). Then, puncturing this disk gives that $\alpha, \beta$ fill $\Sigma_{g,1}$. For the remaining $2k$ punctures, perform the double bigon method $k$ times; each time will increase $i(\alpha,\beta)$ by $2$ and will result in $\alpha,\beta$ filling $\Sigma_{g,2k+1}$ with intersection number \[i(\alpha,\beta)+2k=(2g-1)+2k=2g-2+n.\] For $n=2k$ even, the same argument generalizes if there exists a filling pair $\alpha,\beta$ on $\Sigma_{g,0}$ intersecting $2g$ times; we refer the reader to \cite{aougab-taylor}, Lemma 3.1, for the construction of such a pair. 
\end{proof}

A similar application of the double bigon method gives minimal intersection numbers for $\Sigma_{g,n}$ for $g=0, 2$ (see \cite{aougab-taylor}, Lemma 3.1 and \cite{jeffreys} Theorem 3.3). We summarize the results as follows: 
\begin{center}
    \def\arraystretch{1.5}
    \begin{tabular}{|c|c|c|c|}  
        \hline
        Genus & Punctures  &
        $i(\alpha,\beta)$ \\
        \hline\hline
       $g = 0$ & $n \geq 4$ even & $n-2$ \\ \hline 
       $g = 0$ & $n \geq 4$ odd & $n-1$ \\ \hline 
       $g = 2$ & $n \leq 2$ & $4$ \\ \hline 
       $g = 2$ & $n > 2$ & $2g+n-2$ \\ \hline 
    \end{tabular}
\end{center}

\begin{rem}\label{no n less than 4}
    The case $g=0$, $n<4$ is not considered because the filling curves have intersection number zero: if there two or fewer punctures then a single curve fills and if there are exactly three punctures then the filling pair does not intersect. 
\end{rem}
\noindent\textit{proof of Corollary \ref{main corollary}}
The proof follows immediately from plugging in the values from Proposition \ref{intersection numbers} into the matrices in Theorem \ref{min dilatation} and applying Thurston's construction. Fix a surface $\Sigma_{g,n}$, $g\neq 0,2$ and let $\{\alpha,\beta\}$ be a minimally intersecting filling pair, so that $i(\alpha,\beta)=i_{g,n}$.

Letting $A=\bsM{1 & -i(g,n) \\ 0 & 1}$ and $B=\bsM{1 & 0 \\ i(g,n) & 1}$, by Thurston's Construction (Theorem \ref{thurston construction}) the Thurston pA maps in $\Gamma_{\alpha,\beta}\subset\mcg(\Sigma_{g,n})$--the subset of the mapping class generated by Dehn twists $T_\alpha,T_\beta$ about the curves $\alpha$ and $\beta$--correspond to the hyperbolic elements of  $\Lambda_{i(\alpha,\beta)}=\inr{A,B}$. Moreover, the spectral radius of the elements of $\Lambda_{i(\alpha,\beta)}$ correspond to the dilatation of the pA maps. 

By Theorem \ref{min dilatation}, the minimal nonzero spectral radius in $\Lambda_{i(\alpha,\beta)}$ (and thus minimal dilatation in $\Gamma_{\alpha,\beta}$) is given by 
\[
\frac{1}{2}\left(i(\alpha,\beta)^2+i(\alpha,\beta)\sqrt{i(\alpha,\beta)^2-4}-2\right)
\]
achieved by the hyperbolic matrix
\[
    AB=\bM{1 & -i(\alpha,\beta) \\ 0 & 1}\bM{1 & 0 \\ i(\alpha,\beta) & 1}
    =\bM{1-i(\alpha,\beta)^2 & -i(\alpha,\beta) \\ i(\alpha,\beta) & 1}.
\]
By Thurston's Construction, $AB$ represents the pA map $T_\alpha T_\beta$, the product of Dehn twists about $\alpha$ and $\beta$. The specific values for dilatation in Corollary \ref{main corollary} are obtained by substituting the corresponding values of $i_{g,n}$ from Proposition \ref{intersection numbers} for $i(\alpha,\beta)$.
\qed\newline
\begin{rem}\label{no n equals 2}
    We note that the value of $n =2$ is never realized for any $\Sigma_{g,n}$ justifying the exclusion of this value in Proposition \ref{intersection numbers}. 
\end{rem}
\section{Future Directions}
Throughout this paper we exclusively explored the case where $A$ and $B$ are single curves $\alpha$ and $\beta$, respectively, but the problem of finding the minimal dilatation Thurston pA map extends to the general case of \emph{multicurves} $A=\{\alpha_1,\ldots,\alpha_k\},~ B=\{\beta_1,\ldots,\beta_{\ell}\}$ on $\Sigma_{g,n}$ (i.e. disjoint collections of simple closed curves). The \emph{multitwist} about $A$ and $B$ are the products $\prod_{i=1}^n T_{\alpha_i}, \prod_{i=1}^m T_{\beta_i}$, respectively. We recall that the Thurston construction generalizes for multicurve systems which fill $\Sigma_{g,n}$ to obtain a representation $\rho: \Gamma_{A,B} \rightarrow \PSL_2(\bbR)$ given by 
\[T_A \mapsto \bM{1 & -\mu^{1/2} \\ 0 & 1} \quad T_B \mapsto \bM{1 & 0 \\ \mu^{1/2} & 1},\] 
where $\mu$ is the square of the largest singular value of the $k\x\ell$ intersection matrix $N$ whose $(n,m)$ entry is given by
\[
    N_{n,m}=i(\alpha_n,\beta_m),
\]
i.e., $\mu$ is the \emph{Perron-Frobenius eigenvalue} of $N^{\T}N$ (note that we must work with this matrix instead of $N$ since the latter is not necessarily square). We refer the reader to \cite{primer}, Section 14.1.2 for some background on Perron-Frobenius theory. Leininger \cite{Leininger_2004} derived several useful facts regarding minimal pseudo-Anosov dilatation elements in groups generated by multitwists. 

However, as we noted after stating Corollary \ref{main corollary}, twists about two filling curves become less representative of the entire mapping class group with increasing genus. Thus another question to ask is whether $\lambda(\Gamma_{A,B})$ also increases monotonically with genus, particularly when $A, B$ each consist of $g$ multicurves, i.e. are twists about $2g$ filling curves more characteristic of $\Sigma_{g,n}$, particularly for large $g$.

\bibliographystyle{plain}
\bibliography{refs}
\end{document}